\documentclass{article}
\usepackage[colorlinks=true, allcolors=blue]{hyperref}
\usepackage{array}
\usepackage{amscd}
\usepackage{amsmath}
\usepackage{amssymb}
\usepackage{amsthm}
\usepackage{mathtools}
\usepackage{mathabx}
\usepackage{verbatim}
\usepackage[francais,english]{babel}
\usepackage[utf8]{inputenc}
\usepackage[letterpaper,top=2cm,bottom=2cm,left=3cm,right=3cm,marginparwidth=1.75cm]{geometry}      
\usepackage[pdftex]{graphicx}				
\usepackage{setspace}
\usepackage{physics}
\usepackage{subcaption}
\usepackage[numbers]{natbib}
\usepackage{pdfpages}
\usepackage{bm}
\usepackage{tikz-cd}
\usepackage{wrapfig}
\usepackage{comment}

\begin{document}

\newcommand{\mon}{\textbf{Mon}}
\newcommand{\ab}{\textbf{Ab}}
\newcommand{\semiadd}{\mathcal{E}}
\newcommand{\rep}{Rep(\mathcal{E})}
\newcommand{\repop}{Rep(\mathcal{E}^{op})}
\newcommand{\treillis}{\mathcal{L}}
\renewcommand{\thefootnote}{\fnsymbol{footnote}}
\newcommand{\closetwoheadrightarrow}[1]{\overset{\raisebox{-1ex}{\scriptsize $#1$}}{\twoheadrightarrow}}

\theoremstyle{plain}
\newtheorem{theoreme}{Theorem}[section]

\newtheorem{definition}[theoreme]{Definition}
\theoremstyle{proposition}
\newtheorem{proposition}[theoreme]{Proposition}
\newtheorem*{theoremeA}{Theorem A}
\newtheorem*{propositionA}{Proposition A}
\newtheorem{lemme}[theoreme]{Lemma}
\newtheorem{corollaire}[theoreme]{Corollary}
\theoremstyle{remark}
\newtheorem*{remarque}{Remark}
\newtheorem*{exemples}{Examples}

\author{Benachir El Allaoui\footnotemark}

\date{}

\title{(Co)homological vanishing for non-additive representations of a semi-additive category}

\maketitle

\footnotetext{* Université Sorbonne Paris Nord, Laboratoire Analyse, Géométrie et Applications, CNRS (UMR 7539), 93430, Villetaneuse, France.}

\selectlanguage{english}

\begin{abstract}
We show the vanishing of higher extension groups and torsion groups between linearisation of additive functors from a semi-additive category satisfying some conditions to a category of modules. In particular, we apply our results to the category of correspondence functors of Bouc-Thévenaz.
\end{abstract}

\selectlanguage{french}

\begin{abstract}
On montre l'annulation de groupes d'extensions supérieurs et de groupes de torsion supérieurs entre linéarisations de foncteurs additifs depuis une catégorie semi-additive vérifiant certaines hypothèses vers une catégorie de modules. En particulier, nous appliquons nos résultats à la catégorie des foncteurs de correspondances de Bouc-Thévenaz.
\end{abstract}

\selectlanguage{english}

\tableofcontents

\section{Introduction}

The category of correspondences (or category of relations) is the category whose objects are finite sets and morphisms are binary relations. Functor categories over the category of correspondences have been studied by Bouc-Thévenaz (see \cite{BT2}, \cite{BT1}), they established several finiteness properties and classified the simple functors. In this article, we compute extension groups between linearizations of additive functors from a semi-additive category in a more general framework than that of correspondences, using methods different from those employed by Bouc-Thévenaz. Recall that a semi-additive category is a category that admits finite products and finite coproducts which coincide (as in an additive category, but without assuming that the Hom-sets are abelian groups). We prove the following theorem for additive functors $A, B$ from a semi-additive category $\semiadd$ to the category of commutative monoids $\mon$ :

\begin{theoremeA}[Theorem 5.13]
    Let $k$ be a commutative ring with unit, if $\semiadd$ is $k$-trivial (see definition 5.2) and if $B$ takes finite values, then for all $n \geq 1$, $\textup{Ext}^n_{\mathcal{F}(\semiadd;k)}(k[A],k[B]) = 0$.
\end{theoremeA}

Let $\mathcal{C}$ denote the category of correspondences, $\mathcal{C}$ is $k$-trivial (see remark after definition 5.2). The category $\mathcal{C}$ is semi-additive, with direct sum given by the disjoint union of sets. $\mathcal{C}$ is generated (under direct sum) by the singleton $*$. Hence, any additive functor $A$ from $\mathcal{C}$ to $\mon$ is determined by its value $A(*)$, which carries a semilattice structure (i.e. every element in $A(*)$ is an idempotent, see beginning of section 4), and we obtain an equivalence between the category of additive functors $\textup{Add}(\mathcal{C};\mon)$  and the full subcategory of $\mon$ consisting of monoids such that the sum induces a semilattice structure.

If $T$ is a finite lattice, we can thus regard it as an additive functor from $\mathcal{C}$ to $\mon$, and we write $F_T := k[T]$. From Theorem A we derive the following proposition:

\begin{propositionA}
    Let $T, T'$ be lattices with $T'$ finite, for all $n \geq 1$, $\textup{Ext}^n_{\mathcal{F}(\mathcal{C};k)}(F_T,F_{T'}) = 0$.
\end{propositionA}

For the rest, fix a commutative ring with unit $k$. If $\semiadd$ is a semi-additive category, the sets of morphisms can be endowed canonically with a commutative monoid structure.
The category of functors from $\semiadd$ to $k$-modules, denoted
$\mathcal{F}(\semiadd;k)$, is a Grothendieck category with enough projectives and injectives, and one can perform homological algebra in it just like in a category of modules. A tensor product can be defined that generalizes the usual tensor product of modules and we can define Tor and Ext groups.

If $\semiadd$ is additive, all representable functors $\textup{Hom}_{\semiadd}(a,-)$ take values in the category of abelian groups $\ab$.
The category of additive functors from $\semiadd$ to $\ab$, denoted
$\textup{Add}(\semiadd;\ab)$, is abelian and has enough projectives.
Hence one can use the Dold–Kan correspondence to define a simplicial projective resolution of any functor $A \in \textup{Add}(\semiadd;\ab)$, that is, a simplicial object $P_{\bullet}$ of $s\textup{Add}(\semiadd;\ab)$ that is projective in each degree and induces a resolution of $A$ in $Ch_{\geq 0}(\textup{Add}(\semiadd;\ab))$.
After linearization, $k[P_{\bullet}]$ is again a simplicial projective resolution of $k[A]$. For linearization of additive functors, we may thus assume that projective resolutions are of the form
$k[P_{\bullet}]$.
This allows to reduce the computation of extension groups to the computation of the homology of simplicial abelian groups with coefficients in $k$, for which results are known (see Proposition 3.5 of \cite{DT}).

However, if $\semiadd$ is only semi-additive, the representables take values in the category of commutative monoids $\mon$, and we must work with the category $\textup{Add}(\semiadd;\mon)$, which is not abelian, so the Dold–Kan correspondence no longer applies. To address this issue, in Section 3 we define a model category structure on the simplicial category $s\textup{Add}(\semiadd;\mon)$ in order to construct simplicial projective resolutions of any functor $A$ in
$\textup{Add}(\semiadd;\mon)$. Similarly, we show that $k[A]$ admits a projective resolution of the form $k[P_{\bullet}]$. Thus, the computation of extension groups reduces to computing the homology of simplicial commutative monoids. \newline
In Section 4, we consider the case where $\forall a,b \in \semiadd$, the monoid $\textup{Hom}_{\semiadd}(a,b)$ is a semilattice, that is, it admits a partial order such that every pair of morphisms has a least upper bound given by the monoid law. In this case, we work with simplicial semilattices and show that, since in each degree we have a semigroup whose all elements are idempotent (see the beginning of Section 4), their homotopy groups in  positive degrees vanish; in other words, every simplicial semilattice is discrete. From this, we deduce the vanishing of the extension groups. \newline
In Section 5, we move to a more general framework of $k$-trivial semi-additive categories (see definition 5.2).  In particular, $\forall a,b \in \semiadd$, $\textup{Hom}_{\semiadd}(a,b)$ is a commutative inverse monoid.
Recall that a monoid is inverse if every element admits a unique pseudo-inverse; that is, for all $x$, there exists a unique $x^*$ such that $x = xx^*x$, $x^*=x^*xx^*$. Thus we are reduced to computing the homology of simplicial inverse monoids. 
To this end, we show that they satisfy a kind of Eckmann–Hilton principle but without units. More precisely, if $M_{\bullet}$ is a simplicial inverse semigroup and $e \in M_0$ is an idempotent, the law on $M_{\bullet}$ induces a new one on the homotopy groups $\pi_n(M_{\bullet},e)$ for $n \geq 1$, and we show that it coincides with the usual group law of homotopy groups. Using this, we prove that multiplication by an idempotent induces the identity on homotopy (and hence also on homology). This allows us to show that any simplicial inverse monoid is weakly equivalent to a disjoint union of connected simplicial inverse semigroups, and that group completion induces an isomorphism on homology for these connected simplicial inverse semigroups (in the monoid case, this is simply the classical group completion theorem, see \cite{Quillen2}, though here we don't have necessarily unit).
From all this, we directly deduce the vanishing of the extension groups. \newline
Finally, in the last section, we apply our results to the category of generalized correspondence functors introduced by C. Guillaume (see \cite{CG}). \newline
One can define, in the same way as in the additive case a notion of polynomial functor (for example by using cross-effects) on $\mathcal{F}(\semiadd;k)$ and the results of this article will be used to study the properties of the category $\mathcal{F}(\semiadd;k)$ and the structure of simple functors, when $\mathcal{F}(\semiadd;k)$ has no non-constant polynomial functors. \newline
Let us introduce some notation that we will use in the sequel of the article.
\\

\textbf{Notation}

\begin{itemize}
    \item Throughout, $k$ will denote a commutative unital ring;
    \item If $X$ is a set, $k[X]$ will denote the free $k$-module with basis $X$;
    \item If $X$ is a set, $\mathbb{N}[X]$ will denote the free commutative monoid generated by $X$;
    \item $\mon$ will denote the category of commutative monoids;
    \item If $k$ is a field, we denote by $V^\vee$ the dual of a vector space $V$ over $k$;
    \item If $\mathcal{C}$ is a category, we denote by $\mathcal{C}(x,y)$ the set of morphisms from $x$ to $y$;
    \item If $\mathcal{C}$ is a category, we denote by $s\mathcal{C}$ the category of simplicial objects in $\mathcal{C}$;
    \item If $M$ is a monoid, we denote by $M^\times$ the group of its invertible elements;
    \item If $M$ is a monoid, we denote by $k[M]$ the monoid algebra;
    \item If $M$ is a monoid, we denote by $E(M)$ the set of its idempotents;
    \item If $M$ is a monoid, we denote by $M^{+}$ its group completion.
\end{itemize}

\section{Background on functor categories}

Let $\mathcal{C}$ be a small category, the category of functors $\mathcal{F}(\mathcal{C};k)$ from $\mathcal{C}$ to the category of $k$-modules is complete and cocomplete, limits and colimits are computed pointwise. The category $\mathcal{F}(\mathcal{C};k)$ is abelian and filtered colimits are exact. The functors $P_a^{\mathcal{C}} := k[\mathcal{C}(a,-)]$ defined for all $a \in \mathcal{C}$ form a family of projective generators of $\mathcal{F}(\mathcal{C};k)$. Exactness is detected pointwise, i.e., a sequence $0 \xrightarrow{} F \xrightarrow{} G \xrightarrow{} H \xrightarrow{} 0$ in $\mathcal{F}(\mathcal{C};k)$ is exact if and only if, $\forall x \in \mathcal{C}$, $0 \xrightarrow{} F(x) \xrightarrow{} G(x) \xrightarrow{} H(x) \xrightarrow{} 0$ is exact. The category $\mathcal{F}(\mathcal{C};k)$ has enough projectives, and one can perform homological algebra in it just as in a module category.
We can also define a tensor product generalizing the usual tensor product of modules,
\begin{equation*}
    - \otimes_{\mathcal{C}} - : \mathcal{F}(\mathcal{C}^{op};k) \times \mathcal{F}(\mathcal{C};k) \xrightarrow{} k-\textup{Mod}
\end{equation*}
defined by :
\begin{equation*}
    A \otimes_{\mathcal{C}} B := \int^{c \in \mathcal{C}} A(c) \otimes_{k} B(c)
\end{equation*}
$- \otimes_{\mathcal{C}} -$ commutes with colimits in each variable, in particular, it is right exact, and one can therefore compute its left derived functors, denoted Tor$_*^{\mathcal{C}}(-,-)$. 
\newline
$- \otimes_{\mathcal{C}} -$ is a balanced bifunctor, and hence deriving it with respect to the first or the second variable yields (up to isomorphism) the same result.
\newline
Moreover, for all $ A \in \mathcal{F}(\mathcal{C}^{op};k)$, $B \in \mathcal{F}(\mathcal{C};k)$ and $M \in k-\textup{Mod}$, there is a natural isomorphism
\begin{equation*}
    \textup{Hom}_{k}(A \otimes_{\mathcal{C}} B, M) \cong \textup{Hom}_{\mathcal{F}(\mathcal{C};k)}(B, \textup{Hom}_k(A, M))
\end{equation*}
And for any object $a$ of $\mathcal{C}$, we have natural isomorphisms,
\begin{center}
    $P_a^{\mathcal{C}^{op}} \otimes_{\mathcal{C}} B \cong B(a)$; $A \otimes_{\mathcal{C}} P_a^{\mathcal{C}} \cong A(a)$.
\end{center}
We can also define $- \otimes_{\mathcal{C}} -$ equivalently using the above isomorphisms or by specifying its effect on the projectives generators and the fact that it commutes with colimits in each variable.
\newline
We now recall the definition of semi-additive categories.
\begin{definition}
    A category $\semiadd$ is semi-additive if it has finite products and coproducts and if, for every finite set $I$ and every family of objects $(x_i)_{i \in I}$ of $\semiadd$, the canonical morphism $\coprod_{i \in I} x_i \to \prod_{i \in I} x_i$ is an isomorphism. In this case, the sets of morphisms of $\semiadd$ can be canonically endowed with a commutative monoid structure.
\end{definition}
Throughout this paper, \textbf{$\semiadd$ denotes a semi-additive category}. \newline
We write $\rep$ for the category of additive functors, that is, functors commuting with finite direct sums, from $\semiadd$ to the category of commutative monoids $\mon$. A family of projective generators of $\rep$ is given by the representable functors $\semiadd(x,-)$. We can similarly define a tensor product on $\rep$, but first we recall the definition of the tensor product of commutative monoids.

\begin{definition}
 Let $A,B$ be two commutative monoids, their tensor product $A \otimes_{\mathbb{N}} B$ is defined as the quotient of $\mathbb{N}[A \times B]$ by the relations :
 \begin{itemize}
     \item $(a + a',b) \sim  (a,b) + (a',b)$
     \item $(a,b + b') \sim (a,b) + (a,b')$
     \item $(0,b) \sim (0,0)$
     \item $(a,0) \sim (0,0)$
 \end{itemize}
\end{definition}

As in the case of the tensor product over a commutative ring, it is commutative and associative up to isomorphism.
\newline
If $A \in \repop$, $B \in \rep$, we define :
\begin{equation*}
    A \otimes_{\semiadd} B := \int^{c \in \semiadd} A(c) \otimes_{\mathbb{N}} B(c)
\end{equation*}
This tensor product also commutes with colimits in each variable. Moreover, for every object $a$ of $\semiadd$, there are natural isomorphisms,
\begin{center}
    $\semiadd(-,a) \otimes_{\semiadd} B \cong B(a)$; $A \otimes_{\semiadd} \semiadd(a,-) \cong A(a)$.
\end{center}
A functor $A \in \rep$ is finitely generated, if and only if, it is a quotient of a representable functor $\semiadd(x,-)$ for some $x \in \semiadd$.

\section{Model structure on $s\rep$}

A first tool that we will need later on is a model structure on the category $s\rep$.
To compute torsion groups, we will need to use projective resolutions of functors in $\rep$, and we will use the category $s\rep$ to define them.

Since $\rep$ is complete and cocomplete, we can endow $s\rep$ with the structure of a simplicial category by defining, for two simplicial objects $A,B \in s\rep$, the mapping space $\bold{Hom}_{s\rep}(A,B)_n = \textup{Hom}_{s\rep}(A \otimes \Delta^n_{\bullet}, B)$ where $A \otimes \Delta^n_{\bullet} := \underset{\Delta^n_{\bullet}}{\coprod} A$ and $\Delta^n_{\bullet}$ is the standard $n$-simplex (see Section 4 of \cite{GJ}).
Let $f \in s\rep$, we say that $f$ is a weak equivalence (resp. fibration) if for every object $x$ of $\semiadd$, the morphism $\bold{Hom}_{s\rep}(\semiadd(x,-),f)$ is a weak equivalence (resp. fibration) in $sSet$ with respect to the Quillen model structure. The cofibrations are defined by the lifting property. For every $x \in \semiadd$, a direct application of the Yoneda lemma shows that $\textup{Hom}_{\rep}(\semiadd(x,-), -)$ commutes with filtered colimits.
Thus, $\{\semiadd(x,-)\}_{x \in \semiadd}$ form a system of small projective generators, and therefore, by Theorem 6.9 of \cite{GJ} (originally proved by Quillen in \cite{Quillen}, but in the case of a single small object), the cofibrations, fibrations, and weak equivalences defined above endow $s\rep$ with a model category structure.

\begin{definition}
    Let $B \in \rep$, we denote by $B_\bullet \in s\rep$ the constant simplicial object at $B$. A simplicial projective resolution of $B$ is a simplicial object $P_\bullet$ such that $\forall n \geq 0$, $P_n$ is projective, together with a morphism $P_\bullet \xrightarrow{} B_\bullet$ which is a weak equivalence.
\end{definition}

Simplicial projective resolutions correspond to cofibrant replacements for the model structure on $s\rep$ that we have just defined.
Hence, we immediately obtain the following lemma :

\begin{lemme}
    Let $B \in \rep$, then $B$ admits a simplicial projective resolution.
\end{lemme}

\begin{proposition}
    Let $B \in \rep$ and $P_\bullet \xrightarrow{\sim} B_\bullet$ be a simplicial projective resolution. Then $k[P_\bullet]$ is a projective resolution of $k[B]$ in $\mathcal{F}(\semiadd;k)$.
\end{proposition}

\begin{proof}
    Let $x \in \semiadd$, the isomorphism $\textup{Hom}_{s\rep}(\semiadd(x,-) \otimes \Delta^n_{\bullet},P_\bullet) \cong P_n(x)$, which sends a morphism $\alpha$ to $\alpha_{01...n,x}(Id_x)$ where $\alpha_{01...n}$ is the component indexed by the $n$-simplex $01...n$ of $\Delta^n_{\bullet}$ induces an isomorphism of simplicial objects $\bold{Hom}_{s\rep}(\semiadd(x,-),P_\bullet) \cong P_\bullet(x)$.
    We thus obtain a weak equivalence $P_\bullet(x) \xrightarrow{\sim} B_\bullet(x)$, and hence, after linearization, a quasi-isomorphism $k[P_\bullet(x)] \xrightarrow{\sim} k[B_\bullet(x)]$.
    Therefore, $k[P_\bullet(x)]$ is a projective resolution of $k[B(x)]$, and since exactness is detected pointwise, $k[P_\bullet]$ is indeed a projective resolution of $k[B]$.
\end{proof}

The above results will allow us, when computing derived functors, to consider projective resolutions in $\mathcal{F}(\semiadd;k)$ that are obtained by linearizing projective resolutions in $\rep$.
This additional structure will (under certain assumptions on the category $\semiadd$) enable us to compute them.

The following proposition shows that the linearization functor $k[-]$ commutes with the tensor product of additive functors.

\begin{proposition}
    Let $A \in \repop$, $B \in \rep$, there exists a natural isomorphism in $A$ and $B$,
    \begin{equation*}
    k[A] \otimes_{\semiadd} k[B] \cong k[A \otimes_{\semiadd} B]
    \end{equation*}
\end{proposition}

\begin{proof}
We have a morphism
\begin{equation*}
\alpha : k[A] \otimes_{\semiadd} k[B] \xrightarrow{} k[A \otimes_{\semiadd} B]
\end{equation*}
sending the class of an element $[x] \otimes [y]$ to the element $[x \otimes y]$, where $x \in A(i,i)$, $y \in B(i,i)$, which is natural in $A$ and $B$. Let us show that this is an isomorphism.
\newline
If $B = \semiadd(x,-)$ is a representable, we have the following commutative diagram :
\[
        \begin{tikzcd}
        k[A] \otimes_{\semiadd} k[B] \arrow{r} \arrow{d} & k[A \otimes_{\semiadd} B] \arrow{d} \\
        k[A(x)] \arrow[equal]{r} & k[A(x)]
        \end{tikzcd}
\]
where the first horizontal morphism is the one defined above, and both vertical morphisms are isomorphisms.
\newline
If $B$ is finitely generated projective, it is a retract of a representable $\semiadd(x,-)$, and we thus have the following commutative diagram :
\[
        \begin{tikzcd}
        k[A] \otimes_{\semiadd} k[B] \arrow{r} \arrow[hookrightarrow]{d} & k[A \otimes_{\semiadd} B] \arrow[hookrightarrow]{d} \\
        k[A] \otimes_{\semiadd} k[\semiadd(x,-)] \arrow["\cong"]{r} \arrow[twoheadrightarrow]{d} & k[A \otimes_{\mathcal{E}} \semiadd(x,-)] \arrow[twoheadrightarrow]{d} \\ 
        k[A] \otimes_{\semiadd} k[B] \arrow{r} & k[A \otimes_{\semiadd} B]
        \end{tikzcd}
\]
This shows that $\alpha$ is both an epimorphism and a monomorphism.
\newline
If $B$ is projective, it can be written as the filtered colimit of its finitely generated projective subfunctors, and since the functors $k[A] \otimes_{\semiadd} -$ and $k[A \otimes_{\semiadd} -]$ commute with filtered colimits, it follows again that $\alpha$ is an isomorphism.
\newline
If $B$ is arbitrary, let $P_{\bullet} \xrightarrow{\sim} B_{\bullet}$ be a simplicial projective resolution. We have the following commutative diagram :
\[
        \begin{tikzcd}
        k[A] \otimes_{\semiadd} k[P_{\bullet}] \arrow{r} \arrow{d} & k[A \otimes_{\semiadd} P_{\bullet}] \arrow{d} \\
        k[A] \otimes_{\semiadd} k[B_{\bullet}] \arrow{r} & k[A \otimes_{\semiadd} B_{\bullet}]
        \end{tikzcd}
\]
The first horizontal arrow is an isomorphism of simplicial $k$-modules by what precedes. Both vertical arrows induce isomorphisms on $\pi_0$ because $k[A] \otimes_{\semiadd} -$ and $k[A \otimes_{\semiadd} -]$ commute with coequalizers. Therefore, the second horizontal arrow also induces an isomorphism on $\pi_0$, i.e., $\alpha$ is an isomorphism.
\end{proof}

To show that $k[-]$ also commutes with Hom, we need the following lemma.

\begin{lemme}
Let $A,M \in \textup{\mon}$ and $\pi : M \xrightarrow{} A$ a surjection. Then we have the following exact sequences $k[(M \times_{\pi} M)] \xrightarrow{} k[M] \xrightarrow{} k[A] \xrightarrow{} 0$ and $0 \xrightarrow{} k[\textup{Hom}_{\textup{\mon}}(A,B)] \xrightarrow{} k[\textup{Hom}_{\textup{\mon}}(M,B)] \xrightarrow{} k[\textup{Hom}_{\textup{\mon}}(M \times_{\pi} M,B)]$ for all $B \in \textup{\mon}$.
\end{lemme}

\begin{proof}
Let $X_\bullet \in sSet$ such that for all $n \geq 0$, $X_n = M \times_{\pi} \dots \times_{\pi} M = \{(x_0,...,x_n) \in M^{n+1} | \pi(x_i) = \pi(x_j), \forall i,j\}$ where $d_i(x_0,...,x_n) = (x_0,...,\hat{x_i},...,x_n)$ (we omit the $i^{th}$ coordinate) and $s_j(x_0,...,x_n) = (x_0,...,x_j,x_j,...,x_n)$.
Let $s : A \xrightarrow{} M$ be a set-theoretic section of $\pi$. For all $n \geq 0$ and $i \in \{0,...,n\}$, define $h_i^n : X_n \rightarrow{} X_{n+1}$ sending $(m_0,...,m_n)$ to $(s\circ\pi(m_0),...,s\circ\pi(m_i),m_i,...,m_n)$.
The maps $h_i^n$ form a simplicial homotopy between $id_{X_\bullet}$ and the constant map $(m_0,..,m_n) \mapsto (s\circ \pi(m_0),...,s\circ \pi(m_0))$.
Hence, the canonical morphism $X_\bullet \xrightarrow{} \pi_0(X)_\bullet \cong A_\bullet$ is a homotopy equivalence.
Dually, the morphisms $\textup{Hom}_{\mon}(h_i^n,B)$ form a cosimplicial homotopy (see \cite{Meyer}, Definition 2.1) between the identity and a constant map.
Therefore, $\textup{Hom}_{\mon}(X_\bullet,B) \xrightarrow{} \textup{Hom}_{\mon}(A,B)_\bullet$ is a homotopy equivalence.
After linearization, we obtain that $k[X_\bullet]$ is a resolution of $k[A]$, and $k[\textup{Hom}_{\mon}(X_\bullet,B)]$ is a resolution of $k[\textup{Hom}_{\mon}(A,B)]$.
The beginning of these two resolutions gives the exact sequences of the statement.
\end{proof}

\begin{remarque}
The exact sequences are moreover functorial, so we obtain the same result when replacing $A,M,B$ by functors in $\rep$.
\end{remarque}

We thus deduce the following proposition.

\begin{proposition}
Let $A,B \in \rep$, the canonical morphism
    \begin{align*}
        \Theta_{A,B} : k[\textup{Hom}_{Rep(\semiadd)}(A,B)] \xrightarrow{} \textup{Hom}_{\mathcal{F}(\semiadd;k)}(k[A],k[B])
    \end{align*}
is injective. And if $A$ is finitely generated, it is bijective.
\end{proposition}

\begin{proof}
If $A$ is of the form $\semiadd(x,-)$ for some $x \in \semiadd$, then by the Yoneda lemma we have two $k$-module isomorphisms $k[\textup{Hom}_{\rep}(A,B)] \cong k[B(x)]$ and $\textup{Hom}_{\mathcal{F}(\semiadd;k)}(k[A],k[B]) \cong k[B(x)]$ fitting into the commutative diagram:
    \begin{center}
        \begin{tikzcd}
            k[\textup{Hom}_{Rep(\semiadd)}(A,B)] \arrow{r} \arrow{d}[swap]{\Theta_{A,B}} & k[B(x)] \\
            \textup{Hom}_{\mathcal{F}(\semiadd;k)}(k[A],k[B]) \arrow{ur}
        \end{tikzcd}
    \end{center}
Hence, $\Theta_{A,B}$ is an isomorphism. \\
In general, there exists a surjection $\alpha : \underset{i \in I}{\bigoplus} \: \semiadd(x_i,-) \twoheadrightarrow A$, and we obtain the following commutative diagram :
    \begin{center}
        \begin{tikzcd}
            k[\textup{Hom}_{Rep(\semiadd)}(A,B)] \arrow{r} \arrow{d} & \textup{Hom}_{\mathcal{F}(\semiadd;k)}(k[A],k[B]) \arrow{d} \\
            k[\textup{Hom}_{Rep(\semiadd)}(\underset{i \in I}{\bigoplus} \: \semiadd(x_i,-),B)] \arrow{r} & \textup{Hom}_{\mathcal{F}(\semiadd;k)}(k[\underset{i \in I}{\bigoplus} \: \semiadd(x_i,-)],k[B])
        \end{tikzcd}
    \end{center}
where both vertical arrows are injective, so it suffices to show that the bottom arrow is injective.
Let $0 \neq x \in k[\textup{Hom}_{\rep}(\underset{i \in I}{\bigoplus} \semiadd(x_i,-),B)] \cong k[\underset{i \in I}{\prod} B(x_i)]$, $x$ can be written as $\underset{k = 1}{\overset{n}{\sum}} \lambda_k [\zeta_k]$ where $n \in \mathbb{N}^*$, $\lambda_k \in k \setminus \{0\}$, $\zeta_k \in \underset{i \in I}{\prod} B(x_i)$ and the $\zeta_k$ are pairwise distinct.
Since the sum is finite, there exists a nonempty finite subset $J \subset I$ such that the projections of the $\zeta_k$ in $\underset{j \in J}{\prod} B(x_j)$ remain pairwise distinct.
Hence, the image of $x$ under $k[\textup{Hom}_{\rep}(\underset{i \in I}{\bigoplus} \semiadd(x_i,-),B)] \xrightarrow{} k[\textup{Hom}_{\rep}(\underset{j \in J}{\bigoplus} \semiadd(x_j,-),B)]$ is still nonzero.
The isomorphism $k[\textup{Hom}_{\rep}(\underset{j \in J}{\bigoplus} \semiadd(x_j,-),B)] \cong \textup{Hom}_{\mathcal{F}(\semiadd;k)}(k[\underset{j \in J}{\bigoplus} \semiadd(x_j,-)],k[B])$ then allows us to conclude. \\
If $A$ is finitely generated, there exists a surjection $\alpha : \semiadd(x,-) \twoheadrightarrow A$, and by the previous lemma we have the following commutative diagram in which the rows are exact:
    \begin{center}
        \begin{tikzcd}
            0 \arrow{r} & k[\textup{Hom}_{Rep(\semiadd)}(A,B)] \arrow{r} \arrow{d}{\Theta_{A,B}} & k[\textup{Hom}_{Rep(\semiadd)}(\semiadd(x,-),B)] \arrow{r} \arrow{d}{\Theta_{\semiadd(x,-),B}} & k[\textup{Hom}_{Rep(\semiadd)}(\semiadd(x,-) \times_{\alpha} \semiadd(x,-),B)] \arrow{d}{\Theta_{\semiadd(x,-) \times_{\alpha} \semiadd(x,-),B}} \\
            0 \arrow{r} & \textup{Hom}_{\mathcal{F}(\semiadd;k)}(k[A],k[B]) \arrow{r} & \textup{Hom}_{\mathcal{F}(\semiadd;k)}(k[\semiadd(x,-)],k[B]) \arrow{r} & \textup{Hom}_{\mathcal{F}(\semiadd;k)}(k[\semiadd(x,-) \times_{\alpha} \semiadd(x,-)],k[B])
        \end{tikzcd}
    \end{center}
From what precedes, the three vertical morphisms are injective and the middle one is moreover bijective. Hence, $\Theta_{A,B}$ is bijective.
\end{proof}

In the remainder of this article, we will compute torsion groups, and to deduce results on extension groups, we will need the following proposition :

\begin{proposition}
Let $F \in \mathcal{F}(\semiadd^{op};k)$, $G \in \mathcal{F}(\semiadd;k)$ and $M \in k-\textup{Mod}$.
If $F$ takes projective values, there exists a spectral sequence:
    \begin{center}
        $E^2_{i,j} = \textup{Ext}^j_{k}(\textup{Tor}_i^{\semiadd}(F,G),M) \Longrightarrow \textup{Ext}^{i+j}_{\mathcal{F}(\semiadd;k)}(G, \textup{Hom}_{k}(F,M))$
    \end{center}
In particular, if $\textup{Tor}_i^{\semiadd}(F,G) = 0$ for all $i > 0$, we obtain an isomorphism $\textup{Ext}^j_{k}(F \otimes_{\semiadd} G, M) \cong \textup{Ext}^{j}_{\mathcal{F}(\semiadd;k)}(G, \textup{Hom}_{k}(F,M))$ for all $j > 0$.
\end{proposition}

\begin{proof}
Let $P_\bullet \xrightarrow{} G \xrightarrow{} 0$ be a projective resolution of $G$ and $0 \xrightarrow{} M \xrightarrow{} I^{\bullet}$ an injective resolution of $M$.
We obtain a bicomplex $\textup{Hom}(F \otimes_{\semiadd} P_{i}, I^{j})$.
Filtering first with respect to $i$ and then with respect to $j$, we obtain the page $E^2_{i,j} = \textup{Ext}^j_{k}(\textup{Tor}_i^{\semiadd}(F,G),M)$.
By the adjunction isomorphism, we have an isomorphism of bicomplexes $\textup{Hom}(F \otimes_{\semiadd} P_{i}, I^{j}) \cong \textup{Hom}(P_{i}, \textup{Hom}_{k}(F,I^{j}))$.
Filtering this time with respect to $j$ and then $i$, we obtain the page $'E^2_{i,j} = \textup{Ext}^{i}(G,\textup{Hom}_{k}(F,M))$ if $j = 0$ and $'E^2_{i,j} = 0$ otherwise, since $\textup{Hom}_{k}(F,-)$ is exact.
We then obtain as the limit page $'E^{\infty} =$ $'E^2$. Since both spectral sequences converge to the same limit page, we obtain the desired spectral sequence.
\end{proof}

\section{The case of semilattices}

Let us begin by recalling the definition of a semilattice.

\begin{definition}
Let $M$ be a partially ordered set. If for all $a,b \in M$, the least upper bound $sup(a,b)$ exists, we say that $M$ is a semilattice and we denote $a \vee b := sup(a,b)$. Moreover, if for all $a,b \in M$, the greatest lower bound $inf(a,b)$ exists, we say that $M$ is a lattice and we denote $a \wedge b := inf(a,b)$.
\end{definition}

Equivalently, \textbf{a semilattice can be defined as an abelian semigroup in which every element is idempotent}. The sum of two elements is given by $\vee$, and conversely, by setting $x \leq y$ if $x + y = y$, we obtain a semilattice. This is the point of view that we will use.

Throughout this section, we will assume that for all $x,y \in \semiadd$, $\semiadd(x,y)$ \textbf{is an abelian monoid in which every element is idempotent}. This implies that additive functors as well as their tensor products take values in semilattices :

\begin{lemme}
Let $A \in \rep, B \in \repop$, then $A$ takes values in semilattices and $B \otimes_{\semiadd} A$ is a semilattice.
\end{lemme}

\begin{proof}
There exist $(a_i)_{i \in I}, (b_j)_{j \in J} \in \semiadd^{I \times J}$ and surjections $\oplus_{i \in I} \semiadd(a_i, -) \twoheadrightarrow A$, $\oplus_{j \in J} \semiadd(-,b_j) \twoheadrightarrow B$. For $x \in \semiadd$, $\oplus_{i \in I} \semiadd(a_i, x)$ is an abelian monoid in which every element is idempotent, in particular, it is a semilattice. Hence, the same holds for its quotient $A(x)$. \newline
$B \otimes_{\semiadd} A$ is a semilattice since it is a quotient of $\oplus_{(i,j) \in I \times J} \semiadd(-,b_j) \otimes_{\semiadd} \semiadd(a_i,-) \cong \oplus_{(i,j) \in I \times J} \semiadd(a_i,b_j)$, which is a semilattice by hypothesis on $\semiadd$.
\end{proof}

To show the vanishing of the torsion groups, we will work with simplicial homotopy groups that are easily computable in our case, this is the content of the following lemma. For a definition of the homotopy groups of fibrant simplicial sets, see \cite{GJ}, and more generally, they are defined using a fibrant replacement.

\begin{lemme}
Let $M_{\bullet}$ be a simplicial semilattice, then for all $e \in M_0$ and all $n \geq 1$, $\pi_n(M_{\bullet},e) = 0$.
\end{lemme}

\begin{proof}
Let $n \geq 1$, since the functor $\pi_n$ commutes with finite products, the operation on $M_{\bullet}$ induces a new operation $\pi_n(M_{\bullet},e) \times \pi_n(M_{\bullet},e) \xrightarrow{(\alpha_1, \alpha_2)} \pi_n(M_{\bullet},e^2) = \pi_n(M_{\bullet},e)$. Since $M_n$ is commutative, we have $\alpha_1 = \alpha_2 =: \alpha$, and the associativity of $M_{\bullet}$ gives the following commutative diagram :
\[
        \begin{tikzcd}
        \pi_n(M_{\bullet},e) \times \pi_n(M_{\bullet},e) \times \pi_n(M_{\bullet},e) \arrow["(id \textup{,} (\alpha \textup{,} \alpha))"]{r} \arrow["((\alpha \textup{,} \alpha) \textup{,} id)"]{d} & \pi_n(M_{\bullet},e) \times \pi_n(M_{\bullet},e) \arrow["(\alpha \textup{,} \alpha)"]{d}
        \\
        \pi_n(M_{\bullet},e) \times \pi_n(M_{\bullet},e) \arrow["(\alpha \textup{,} \alpha)"]{r} & \pi_n(M_{\bullet},e)
        \end{tikzcd}
\]
Hence, for all $x,y,z \in \pi_n(M_{\bullet},e)$, we have $\alpha(x) + \alpha(\alpha(y) + \alpha(z)) = \alpha(\alpha(x+y)) + \alpha(z)$. Taking $y = z = 0$, we get $\alpha(x) = \alpha(\alpha(x))$, i.e., $\alpha = \alpha \circ \alpha$. Denoting by $\mu$ the addition in $M_\bullet$, the following commutative diagram :
\[
        \begin{tikzcd}
        M_\bullet \arrow["(id \textup{,} id)"]{r} \arrow["id",swap]{rd} & M_\bullet \oplus M_\bullet \arrow["\mu"]{d} \\
         & M_\bullet
        \end{tikzcd}
\]
implies that $2\alpha = id$, so $\alpha$ is surjective, but from the previous relation, we have $4\alpha = 2\alpha$, hence $\alpha = 0$. Thus, $\pi_n(M_\bullet,e) = 0$.
\end{proof}

Therefore, the morphism $M_\bullet \xrightarrow{} \pi_0(M_\bullet)_\bullet$ induced by $M_0 \xrightarrow{} \pi_0(M_\bullet)$ is a weak homotopy equivalence, where $\pi_0(M_\bullet)_\bullet$ is the constant simplicial set on $\pi_0(M_\bullet)$. We immediately deduce the following proposition :

\begin{proposition}
Let $M_\bullet$ be a simplicial semilattice. Then for all $n \geq 1$, $H_n(M_{\bullet};k) = 0$.
\end{proposition}

\begin{proof}
We know that the morphism $M_\bullet \xrightarrow{} \pi_0(M_\bullet)_\bullet$ induced by $M_0 \xrightarrow{} \pi_0(M_\bullet)$ is a weak equivalence. Hence, $k[M_\bullet] \xrightarrow{} k[\pi_0(M_\bullet)_\bullet]$ is a quasi-isomorphism. Therefore,
\begin{equation*}
        H_n(M_{\bullet};k) = \left\{
    \begin{aligned}
         & k[\pi_0(M_\bullet)] \textup{ if $n = 0$} \\
         & 0 \textup{ if $n \geq 1$} \\
    \end{aligned}
        \right.
    \end{equation*}
\end{proof}

We directly deduce the following theorem :

\begin{theoreme}
Let $A \in \repop, B \in \rep$, then for all $n \geq 1$, $\textup{Tor}_n^{\semiadd}(k[A],k[B]) = 0$.
\end{theoreme}

\begin{proof}
Let $P_\bullet \xrightarrow{\sim} B_\bullet$ be a simplicial projective resolution of $B$, then $k[P_\bullet]$ is a projective resolution of $k[B]$. Hence, $\textup{Tor}_n^{\semiadd}(k[A],k[B]) = H_n(k[A] \otimes_{\semiadd} k[P_\bullet]) \cong H_n(k[A \otimes_{\semiadd} P_\bullet])$. \newline
But by the Lemma 4.2, $A \otimes_{\semiadd} P_\bullet$ is a simplicial semilattice. \newline
Therefore, we can apply Proposition 4.4 to deduce that $\textup{Tor}_n^{\semiadd}(k[A],k[B]) = 0$ for $n \geq 1$.
\end{proof}

If $A \in \rep$, we denote $A^{\#} := \textup{Hom}_{\mon}(-,\{0,1\}) \circ A$. \newline
We have the following duality property:

\begin{proposition}
Let $A \in \rep$ taking finite values, then we have an isomorphism $k[A] \cong k^{A^{\#}}$.
\end{proposition}

\begin{proof}
By assumption, for all $x \in \semiadd$, $A(x)$ is a finite monoid in which every element is idempotent, and therefore $A(x)$ is a finite lattice (see Proposition 2.1 of \cite{Steinberg}). Thus, we have an isomorphism $k[A(x)] \cong k^{A(x)}$ (see \cite{Steinberg}, Corollary 9.5; the proof works the same way for a commutative unitary ring). Moreover, we have an isomorphism $\textup{Hom}_{\mon}(A(x),\{0,1\}) \cong A(x)$ which sends $f$ to $max(f^{-1}(1))$ (see Lemma 2.1 and Theorem 3.6 of \cite{Pirashvili}). Combining these isomorphisms, we obtain $k[A(x)] \cong k^{A^{\#}(x)}$, which sends $a$ to the evaluation at $a$, i.e., the map $\phi \mapsto \phi(a)$. In particular, this isomorphism is functorial in $x$, and we obtain an isomorphism $k[A] \cong k^{A^{\#}}$.
\end{proof}

We can now deduce the vanishing of the Ext groups.

\begin{theoreme}
Let $A,B \in \rep$ such that $B$ takes finite values. Then, for every $n \geq 1$, we have $\textup{Ext}^n_{\mathcal{F}(\semiadd;k)}(k[A],k[B]) = 0$.
\end{theoreme}

\begin{proof}
If $B$ takes finite values, then by Proposition 4.6, we have isomorphisms $k[B] \cong k^{B^{\#}} \cong k[B^{\#}]^{\vee}$ and $\textup{Ext}^n_{\mathcal{F}(\semiadd;k)}(k[A],k[B]) \cong \textup{Ext}^n_{\mathcal{F}(\semiadd;k)}(k[A],k[B^{\#}]^{\vee})$. We can also apply Proposition 3.7, and since $\textup{Tor}_{n}^{\semiadd}(k[B^{\#}],k[A]) = 0$ for all $n \geq 1$, we have isomorphisms \newline $\textup{Ext}^n_{\mathcal{F}(\semiadd;k)}(k[A],k[B^{\#}]^{\vee}) \cong \textup{Ext}^n_{k}(k[B^{\#} \otimes_{\semiadd} A],k) = 0$ for all $n \geq 1$ because $k[B^{\#} \otimes_{\semiadd} A]$ is projective.
\end{proof}

\begin{remarque}
If the category $\mathcal{F}(\semiadd;k)$ is locally noetherian and if $A$ is finitely generated, the result still holds without any assumption on $B$. Indeed, it suffices to write $B$ as a filtered colimit of its finitely generated subfunctors and use the fact that $\textup{Ext}^n_{\mathcal{F}(\semiadd;k)}(A,-)$ commutes with filtered colimits. This is, for instance, the case for the category of generalized correspondence functors, whose definition is recalled in Section 6.
\end{remarque}

\section{The $k$-trivial case}

In this section, we work in a more general setting than that of semilattices, namely that of $k$-trivial categories, which we will now define. Similarly to the previous section, we will reduce the computation of torsion groups to the computation of the homology of a simplicial inverse monoid.

\begin{definition}
Let $M \in \textup{\mon}$, if $M$ is finite and regular, i.e. $\forall m \in M$, $\exists n \in M$ such that $m = m+n+m$, and for every idempotent $e \in M$, $(e + M)^{\times} \otimes_{\mathbb{Z}} k = 0$, we say that $M$ is $k$-trivial.
\end{definition}

\begin{definition}
Let $\mathcal{C}$ be a semi-additive category.
We say that $\mathcal{C}$ is $k$-trivial if $\forall x,y \in \mathcal{C}$, the monoid $\mathcal{C}(x,y)$ is $k$-trivial.
\end{definition}

\begin{remarque}
If $M \in \mon$ is a semilattice, then $M$ is regular and $\forall e \in E(M)$, $(e + M)^{\times} = 0$, hence $M$ is $k$-trivial. In particular, the category of correspondences is $k$-trivial.
\end{remarque}

\begin{remarque}
The conditions on the $G_e$ are necessary to show the Lemma 5.5, without these conditions, if $M$ is $k$-trivial and $k$ is a field, the algebra $k[M]$ is no longer split semisimple and we need this to deduce the Theorem 5.13 from the Theorem 5.11.
\end{remarque}

\begin{definition}
Let $S$ be a semigroup, we say that $S$ is inverse if, for every $x \in S$, there exists a unique element $x^*$ in $S$ such that $xx^*x = x$ and $x^*xx^* = x^*$.
\end{definition}

For the sequel, we will need the following lemma:

\begin{lemme}
Let $M$ be a monoid, then $M$ is inverse if and only if $M$ is regular and its idempotents commute with each other. In particular, if $M$ is $k$-trivial, then $M$ is inverse.
\end{lemme}

\begin{proof}
See Theorem 3.2 of \cite{Steinberg}.
\end{proof}

For a study of finite inverse monoids, see Chapter 3 of \cite{Steinberg}.
An important point is that if $M$ is a finite inverse monoid, the monoid algebra $k[M]$ is isomorphic to the algebra of a groupoid (see Section 8.1 of \cite{Steinberg} for the construction of the algebra of a category).
More precisely, let $G(M)$ be the groupoid whose set of objects is $E(M)$ and such that for $e,f \in E(M)$, $\textup{Hom}_{G(M)}(e,f) = \{ m \in M | m^{*}m = e, mm^{*} = f \}$ (we do not assume $M$ commutative).
There is an isomorphism of $k$-algebras $\alpha : k[M] \xrightarrow{} k[G(M)]$ (see \cite{Steinberg} theorem 9.3 for the construction of $\alpha$).
If $M$ is moreover commutative, the groupoid $G(M)$ is isomorphic to $\underset{e \in E(M)}{\coprod} G_e$, where $G_e = (e + M)^{\times}$, viewed here as a single-object groupoid.
Hence, we obtain an isomorphism of $k$-algebras $k[M] \xrightarrow{\alpha} k[G(M)] \xrightarrow{\sim} \underset{e \in E(M)}{\prod} k[G_e]$.
If $M \in \mon$, we denote $M^* := \textup{Hom}_{\mon}(M,k_{\mu})$, where $k_{\mu}$ denotes the multiplicative monoid underlying $k$.

\begin{lemme}
Let $M$ be a $k$-trivial monoid, and suppose that $k$ contains all roots of unity. Then the evaluation morphism $ev : k[M] \xrightarrow{} k^{M^{*}}$ sending $[m]$ to evaluation at $m$ is an isomorphism.
\end{lemme}

\begin{proof}
Suppose first that $k$ is a field.
By hypothesis, each $G_e = (e + M)^{\times}$ is a finite abelian group whose order is invertible in $k$.
Thus, the morphism $ev_e : k[G_e] \xrightarrow{} k^{\textup{Hom}_{\mathbb{Z}}(G_e,k^{\times})}$ sending $[g]$ to evaluation at $g$ is an isomorphism.
It is straightforward to check that the following diagram is commutative:
\[
    \begin{tikzcd}[]
        k[M] \arrow[r,"ev"] \arrow[d] & k^{M^*} \arrow[d] \\
        \underset{e \in E(M)}{\prod} k[G_e] \arrow[r,"\prod ev_{e}"] & \underset{e \in E(M)}{\prod} k^{\textup{Hom}_{\mathbb{Z}}(G_e,k^{\times})}
    \end{tikzcd}
\]
where the left vertical arrow is just the composition $k[M] \xrightarrow{\alpha} k[G(M)] \xrightarrow{\sim} \underset{e \in E(M)}{\prod} k[G_e]$ and the right vertical arrow is the composition \newline $k^{M^{*}} \cong \textup{Hom}_{Set}(\textup{Hom}_{k-alg}(k[M],k),k) \cong \textup{Hom}_{Set}(\textup{Hom}_{k-alg}(\underset{e \in E(M)}{\prod} k[G_e],k),k) \cong \underset{e \in E(M)}{\prod} k^{\textup{Hom}_{\mathbb{Z}}(G_e,k^{\times})}$ and we deduce that $ev$ is also an isomorphism.
\newline
If $k$ is not a field, then for every morphism of rings $k \xrightarrow{} K$ with $K$ a field, $ev \otimes_{k} K$ is an isomorphism.
Hence we may apply the following lemma to conclude.
\end{proof}

\begin{lemme}
Let $f : V \xrightarrow{} W$ be a morphism of $k$-modules, where $V$ and $W$ are finitely generated projective.
If for every morphism of rings $k \xrightarrow{} K$ with $K$ a field, the morphism $f \otimes_{k} K$ is an isomorphism, then $f$ itself is an isomorphism.
\end{lemme}

\begin{proof}
The cokernel $Coker f$ is finitely generated.
If $Coker f \neq 0$, then there exists a field $K$ and a surjective morphism of $k$-modules $Coker f \twoheadrightarrow K$.
By hypothesis, $f \otimes_{k} K$ is an isomorphism, hence $0 = Coker f \otimes_{k} K \twoheadrightarrow K \otimes_{k} K \neq 0$ which is a contradiction.
Therefore, $Coker f = 0$, and we have a short exact sequence $0 \xrightarrow{} Ker f \xrightarrow{} V \xrightarrow{} W \xrightarrow{} 0$.
Since $W$ is projective, the sequence splits, and $Ker f$ is finitely generated.
Applying the same argument as for $Coker f$, we deduce $Ker f = 0$, so $f$ is an isomorphism.
\end{proof}

If $A \in \rep$, we define similarly $A^* := \textup{Hom}_{\mon}(-,k_{\mu}) \circ A$.
Since the evaluation morphism is functorial, we immediately obtain the following duality property for additive functors :

\begin{proposition}
Let $A \in \rep$ taking finite values, and suppose that $\semiadd$ is $k$-trivial and that $k$ contains all roots of unity.
Then there is an isomorphism $k[A] \cong k^{A^*}$.
\end{proposition}

\begin{proof}
Thanks to lemma 5.5, we only need to check that $A$ takes values in $k$-trivial monoids.
There exists $(a_i)_{i \in I} \in \semiadd^{I}$ and a surjection $\pi : \oplus_{i \in I} \semiadd(a_i,-) \twoheadrightarrow A$.
Since $\semiadd$ is $k$-trivial, for each $i$, $\semiadd(a_i,-)$ takes values in $k$-trivial monoids.
Let $x \in \semiadd$ and $e \in A(x)$ be an idempotent.
There exists an idempotent $(f_i)_{i \in I} \in \oplus_{i \in I} \semiadd(a_i,x)$ mapping to $e$, and $\pi$ induces surjections $((f_i)_i + \oplus_{i \in I} \semiadd(a_i,x))^{\times} \twoheadrightarrow{} (e + A(x))^{\times}$ and $0 = (\oplus_{i \in I} (f_i + \semiadd(a_i,x))^{\times}) \otimes_{\mathbb{Z}} k \cong ((f_i)_i + \oplus_{i \in I} \semiadd(a_i,x))^{\times} \otimes_{\mathbb{Z}} k \twoheadrightarrow (e + A(x))^{\times} \otimes_{\mathbb{Z}} k$.
Hence $A(x)$ is $k$-trivial.
\end{proof}

Following the same spirit as in the previous section, we will work with simplicial inverse semigroups.
If $S_\bullet$ is a simplicial semigroup, its multiplication induces a new operation on its homotopy groups.
In general, this does not coincide with the usual operation on homotopy groups, but if we assume that $S_\bullet$ is both inverse and commutative, then the two coincide, it is the object of the following proposition.

\begin{proposition}
Let $S_\bullet$ be a non-empty commutative simplicial inverse semigroup, $e \in S_0$ an idempotent, and $n \in \mathbb{N}^*$.
Then the induced operation on $\pi_n(S_\bullet,e)$ is the same as its usual one.
\end{proposition}

\begin{proof}
The operation of $S_\bullet$ induces a new operation on homotopy groups,
\newline $\pi_n(S_\bullet,e) \times \pi_n(S_\bullet,e) \xrightarrow{(\gamma,\gamma)} \pi_n(S_\bullet,e)$,
where both components are given by the same morphism $\gamma$, since $S_\bullet$ is commutative.
The associativity of $S_\bullet$ gives the following commutative diagram :
    \[
        \begin{tikzcd}[]
        \pi_n(S_\bullet,e) \times \pi_n(S_\bullet,e) \times \pi_n(S_\bullet,e) \arrow[r,"((\gamma \text{,} \gamma) \text{,} id)"] \arrow[d, "(id \text{,} (\gamma \text{,} \gamma))"] & \pi_n(S_\bullet,e) \times \pi_n(S_\bullet,e) \arrow[d, "(\gamma \text{,} \gamma)"] \\
        \pi_n(S_\bullet,e) \times \pi_n(S_\bullet,e) \arrow[r, "(\gamma \text{,} \gamma)"] & \pi_n(S_\bullet,e)
        \end{tikzcd}
    \]
Hence $\gamma^2 = \gamma$.
Moreover, the morphism $S_\bullet \xrightarrow{} S_\bullet \times S_\bullet$ sending $a$ to $(a,a^* + a)$ is a section of $S_\bullet \times S_\bullet \xrightarrow{+} S_\bullet$.
Thus, $(\gamma, \gamma)$ is an epimorphism, and so is $\gamma$. Therefore, $\gamma = id$.
\end{proof}

\begin{remarque}
We need to assume that $S_\bullet$ is commutative and inverse, otherwise, the result may be false.
Indeed, if $n \in \mathbb{N}^*$ and $V$ is an abelian group, we denote by $K(V,n)$ the $n$-th Eilenberg–Mac Lane space associated to $V$.
The functor $K(-,n)$ commutes with finite products, and there is an isomorphism $\pi_n(K(-,n), *) \cong Id$ between endofunctors of the category of abelian groups.
Therefore, any new operation on $V$ induces one on $\pi_n(K(V,n), *)$, and the above result fails.
For instance, one may consider the product $a.b = a$, which turns $V$ into a semigroup.
\end{remarque}

We then deduce the following proposition:

\begin{proposition}
Let $S_\bullet$ be a connected commutative simplicial inverse semigroup.
Then, for all $n \geq 1$, there is an isomorphism $H_n(S_\bullet;k) \cong H_n(S^+_\bullet;k)$.
\end{proposition}

\begin{proof}
For each integer $n$, set $\Tilde{S}_n = S_n \bigsqcup \{1\}$, the monoid obtained by adjoining a unit to $S_n$.
Then $\Tilde{S}_{\bullet}$ is a commutative simplicial monoid, and we have a morphism of simplicial sets $S_{\bullet} \xrightarrow{} \Tilde{S}_{\bullet}$.
We have the following commutative square, for all $n \geq 0$ :
\[
        \begin{tikzcd}[]
        H_n(S_{\bullet};k) \arrow{r} \arrow{d} & H_n(\Tilde{S}_{\bullet};k) \arrow{d} \\
        H_n(S_{\bullet}^+;k) \arrow{r} & H_n(\Tilde{S}_{\bullet}^+;k)
        \end{tikzcd}
\]
where the vertical maps and the bottom horizontal map come from group completions.
The morphism $S_{\bullet}^+ \xrightarrow{} \Tilde{S}_{\bullet}^+$ is an isomorphism of simplicial abelian groups, hence the bottom map is an isomorphism for all $n \geq 0$.
By construction of $\Tilde{S}_{\bullet}$, we have $H_*(\Tilde{S}_{\bullet};k) \cong H_*(S_{\bullet};k) \oplus k[0]$, and thus the top horizontal map is an isomorphism for all $n \geq 1$.
It remains to show that $H_n(\Tilde{S}_{\bullet};k) \xrightarrow{} H_n(\Tilde{S}_{\bullet}^+;k)$ is an isomorphism for $n \geq 1$.
Since $\Tilde{S}$ is commutative, we may apply the group completion theorem (see Theorem Q.4 \cite{Quillen2}), yielding an isomorphism
$H_*(\Tilde{S}_{\bullet};k)[\pi_0(\Tilde{S}_{\bullet})^{-1}] \cong H_*((\Tilde{S}_{\bullet})^+;k)$.
By hypothesis, $\pi_0(\Tilde{S}_{\bullet}) = \{[1],[e]\}$, where $e \in S_0$ is an idempotent.
Addition by $e$ is given in degree $n$ by the following composition :
\begin{align*}
  S_n &\longrightarrow S_n \times S_n \longrightarrow S_n\\
  a &\longmapsto \bigl(a, s_0 \circ \cdots \circ s_n(e)\bigr)
      \longmapsto a + s_0 \circ \cdots \circ s_n(e)
\end{align*}
By the Propososition 5.8, since $s_0 \circ \cdots \circ s_n(e)$ is null-homotopic, the addition by $e$ induces the identity on $\pi_n$. Thus, addition by $e$ induces an isomorphism in homotopy and therefore also an isomorphism on $H_n(S_\bullet;k)$. Moreover, $H_n(S_\bullet;k) \cong H_n(\Tilde{S}_\bullet;k)$, for all $n \geq 1$, and hence it also induces an isomorphism on $H_n(\Tilde{S}_\bullet;k)$ for $n \geq 1$. Consequently, $H_n(\Tilde{S}_\bullet;k) \cong H_n(\Tilde{S}_{\bullet};k)[\pi_0(\Tilde{S}_{\bullet})^{-1}] \cong H_n(\Tilde{S}_{\bullet}^+;k)$, for all $n \geq 1$.
\end{proof}

We will apply this proposition to deduce the vanishing of the homology of simplicial inverse semigroups. Before that, we need to introduce a new notation. We recall that for a semigroup $M$, the period of an element $x \in M$ is the smallest integer $d > 0$ so that there exists $N \in \mathbb{N}$ such that $\forall n \geq N$, $x^{n + d} = x^{n}$ (see section 1.2 of \cite{Steinberg}). Let $\mathcal{C}$ be the set of commutative semigroups $M$ such that $\forall x \in M$, $x$ is of finite period and its period is coprime to every prime number non-invertible in $k$. It is clear that $\mathcal{C}$ is stable under subsemigroups, direct sums and that if $M \xrightarrow{} N$ is a surjective morphism of semigroups and $M \in \mathcal{C}$ then $N \in \mathcal{C}$.

\begin{remarque}
    From the definition, it is clear that $\mathcal{C}$ contains all k-trivial monoids.
\end{remarque}

We have the following proposition.

\begin{proposition}
If $M_{\bullet}$ is a commutative inverse simplicial monoid such that for all $n \geq 0$, $M_n \in \mathcal{C}$, then $H_n(M_\bullet;k) = 0$ for all $n \geq 1$.
\end{proposition}

\begin{proof}
By decomposing $M_{\bullet}$ as the disjoint union of its connected components, we obtain the following isomorphism:
\begin{equation*}
        H_*(M_{\bullet};k) \cong \oplus_{[x] \in \pi_0(M_{\bullet})} H_*(M_{\bullet}^x;k)
\end{equation*}
where $x \in M_0$ is a representative of $[x]$ and $M_{\bullet}^x$ is the connected component of $x$. Moreover, we have the following commutative diagram,
\newline \begin{tikzcd}
    M_{\bullet}^{x+x^*} \ar[r,"+x"]\ar[rr,out=-30,in=210,swap,"+(x^{*}+x)"] & M_{\bullet}^{x} \ar[r,"+x^*"] & M_{\bullet}^{x+x^*}
    \end{tikzcd}
    and \begin{tikzcd}
    M_{\bullet}^{x} \ar[r,"+x^*"]\ar[rr,out=-30,in=210,swap,"+(x^*+x)"] & M_{\bullet}^{x+x^*} \ar[r,"+x"] & M_{\bullet}^{x}
    \end{tikzcd}.
Let $e := x^*+x$, since $e$ is an idempotent, by Proposition 5.8, the first composite induces the identity in homotopy and thus also in homology. We also have the following diagram (which is not commutative), where $\mu$ denotes the sum :
    \[
        \begin{tikzcd}[]
        M_{\bullet}^{x} \arrow[r,"t \mapsto (t \text{,} e)"] \arrow[dr, swap, "t \mapsto (t \text{,} t+t^*)"] & M_{\bullet}^{x} \times E(M_{\bullet}^{e}) \arrow[r,"\mu"] & M_{\bullet}^{x} \\
        & M_{\bullet}^{x} \times E(M_{\bullet}^{e}) \arrow[ur,"\mu"] &
        \end{tikzcd}
    \]
Since $E(M_{\bullet}^{e})$ is a connected simplicial semilattice, it is contractible. We deduce that the two composites are homotopic, but the first one equals $+e$ and the second equals the identity. Hence, $+e$ induces the identity on $H_*(M_{\bullet}^{x};k)$.
Therefore, we have $H_*(M_{\bullet}^x;k) \cong H_*(M_{\bullet}^{e};k)$. By Proposition 5.9, we also have $H_*(M_{\bullet}^{e};k) \cong H_*((M_{\bullet}^{e})^+;k)$. 
Now, since $M_n \in \mathcal{C}$, $(M_n^{e})^{+}$ is a torsion abelian group such that, for every prime number $p$ such that $(M_n^{e})^{+}$ has $p$-torsion, $p \in k^{\times}$. By decomposing $(M_n^{e})^{+}$ as the direct sum of its $p$-primary components, we see that this implies, $(M_n^{e})^{+} \otimes_{\mathbb{Z}} k = 0$ and $\textup{Tor}_1^{\mathbb{Z}}((M_n^{e})^{+},k) = 0$. It follows that $H_n((M_{\bullet}^{e})^{+};k) = 0$ for all $n \geq 1$ (see Corollary 3.6 of \cite{DT}). Finally, $H_n(M_{\bullet};k) = 0$ for all $n \geq 1$.
\end{proof}

We can now prove the following theorem.

\begin{theoreme}
Suppose that $\semiadd$ is $k$-trivial and let $A \in \repop$ and $B \in \rep$. Then for all $n \geq 1$, $\textup{Tor}_n^{\semiadd}(k[A],k[B]) = 0$.
\end{theoreme}

\begin{proof}
Let $P_{\bullet} \xrightarrow{\sim} B_{\bullet}$ be a simplicial projective resolution of $B$, then $\textup{Tor}_n^{\semiadd}(k[A],k[B]) \cong H_n(k[A \otimes_{\semiadd} P_{\bullet}])$. It remains to show that the simplicial monoid $A \otimes_{\semiadd} P_{\bullet}$ satisfies the assumptions of Proposition 5.10, which is done in the following lemma.
\end{proof}

\begin{lemme}
    Suppose $\semiadd$ is $k$-trivial and let $A \in \repop$, $B \in \rep$, then $A \otimes_{\semiadd} B$ is a commutative inverse monoid in $\mathcal{C}$.
\end{lemme}

\begin{proof}
    There exist $(a_i)_{i \in I}, (b_j)_{j \in J} \in \semiadd^{I \times J}$ and surjections $\oplus_{i \in I} \semiadd(-, a_i) \twoheadrightarrow A$, $\oplus_{j \in J} \semiadd(b_j,-) \twoheadrightarrow B$. We get a surjection $\oplus_{(i,j) \in I \times J} \semiadd(b_j,a_i) \twoheadrightarrow A \otimes_{\semiadd} B$. It suffices to show that $\forall i,j$, $\semiadd(b_j,a_i)$ is a commutative inverse monoid in $\mathcal{C}$. By lemma 5.4, we know that $\semiadd(b_j,a_i)$ is inverse and from the definition of $k$-trivial monoids, it is immediate that $\semiadd(b_j,a_i)$ is in $\mathcal{C}$.
\end{proof}

We thus deduce a vanishing result for extension groups.

\begin{theoreme}
Suppose that $\semiadd$ is $k$-trivial and let $A,B \in \rep$ such that $B$ takes finite values. Then, for all integers $n \geq 1$, we have $\textup{Ext}^n_{\mathcal{F}(\semiadd;k)}(k[A],k[B]) = 0$.
\end{theoreme}

\begin{proof}
Suppose that $k$ contains all roots of unity, so that we can apply Proposition 5.7 to obtain isomorphisms $k[B] \cong k^{B^{*}} \cong k[B^{*}]^{\vee}$ and $\textup{Ext}^n_{\mathcal{F}(\semiadd;k)}(k[A],k[B]) \cong \textup{Ext}^n_{\mathcal{F}(\semiadd;k)}(k[A],k[B^{*}]^{\vee})$. We can also apply Proposition 3.7, and since $\textup{Tor}_{n}^{\semiadd}(k[B^{*}],k[A]) = 0$ for all $n > 0$, we have isomorphisms $\textup{Ext}^n_{\mathcal{F}(\semiadd;k)}(k[A],k[B^{*}]^{\vee}) \cong \textup{Ext}^n_{k}(k[B^{*} \otimes_{\semiadd} A],k) = 0$ for all $n > 0$ because $k[B^{*} \otimes_{\semiadd} A]$ is projective. \newline
In the general case, let $\mathbb{Z}^{cycl}$ be the subring of $\mathbb{C}$ generated by all the roots of unity and define \newline $K := k \otimes_{\mathbb{Z}} \mathbb{Z}^{cycl}$, $K$ contains all the roots of unity. Since $\mathbb{Z}^{cycl}$ is a free $\mathbb{Z}$-module, the morphism $k \xrightarrow{} K$ is split. Since, by Proposition 3.3, the linearization of additive functors admits as a projective resolution the linearization of projective functors, the adjunction isomorphism \newline $\textup{Hom}_{\mathcal{F}(\semiadd;K)}(\mathbb{Z}[A] \otimes_{\mathbb{Z}} K, K[B]) \cong \textup{Hom}_{\mathcal{F}(\semiadd;\mathbb{Z})}(\mathbb{Z}[A], K[B])$ derives to an isomorphism \newline $\textup{Ext}_{\mathcal{F}(\semiadd;K)}^{n}(K[A], K[B]) \cong \textup{Ext}_{\mathcal{F}(\semiadd;\mathbb{Z})}^{n}(\mathbb{Z}[A], K[B])$. The same isomorphism holds with $k$ in place of $K$. Since $k \xrightarrow{} K$ is split, $\textup{Ext}_{\mathcal{F}(\semiadd;\mathbb{Z})}^{n}(\mathbb{Z}[A], k[B])$ is a direct summand of $\textup{Ext}_{\mathcal{F}(\semiadd;\mathbb{Z})}^{n}(\mathbb{Z}[A], K[B])$, which vanishes. Therefore, $\textup{Ext}_{\mathcal{F}(\semiadd;k)}^{n}(k[A], k[B]) = 0$.
\end{proof}

\section{Generalized correspondence functors}

In this section, we will apply the previous results to the category of generalized correspondence functors defined by C. Guillaume \cite{Guillaume}, continuing the work of Bouc–Thévenaz on the category of correspondence functors \cite{BT1}.

\begin{definition}
Let $T$ be a finite distributive lattice. We denote by $\mathcal{C}_T$ the category of generalized correspondences with values in $T$, whose objects are finite sets, and if $X,Y$ are two finite sets, $\textup{Hom}_{\mathcal{C}_T}(X,Y) = T^{X \times Y}$. The composition of $S \in \textup{Hom}_{\mathcal{C}_T}(X,Y)$ and $R \in \textup{Hom}_{\mathcal{C}_T}(Y,Z)$ is given by: $RS(z,x) = \underset{y \in Y}{\bigvee}S(z,y) \wedge R(y,x)$.
\end{definition}

\begin{lemme}
$\mathcal{C}_T$ is a semi-additive category, and if $X,Y \in \mathcal{C}_T$, $\textup{Hom}_{\mathcal{C}_T}(X,Y)$ is a finite lattice.
\end{lemme}

\begin{proof}
It is immediate that the product and coproduct of two objects in $\mathcal{C}_T$ are given by the disjoint union, and that the zero object of $\mathcal{C}_T$ is the empty set.
Let $S \in \textup{Hom}_{\mathcal{C}_T}(X,Y)$ be a morphism, $S+S$ is defined as the following composite:
\[
        \begin{tikzcd}[]
        X \arrow[r, "(id_X \text{,} id_X)"] & X \oplus X \arrow[r, "(S \text{,} S)"] & Y \oplus Y \arrow[r, "(id_Y \text{,} id_Y)"] & Y
        \end{tikzcd}
    \]
We have:
\begin{center}
    $(S,S)(id_X,id_X)(x,y) = \underset{x' \in X \oplus X}{\bigvee} (S,S)(x',y) \wedge (id_X,id_X)(x,x')
    = (S,S)(x,y) = S(x,y)$
\end{center}
And therefore:
\begin{center}
    $(id_Y,id_Y)(S,S)(id_X,id_X)(x,y) = \underset{y' \in Y \oplus Y}{\bigvee} (id_Y,id_Y)(y',y) \wedge (S,S)(id_X,id_X)(x,y')
    = (S,S)(id_X,id_X)(x,y) = S(x,y)$
\end{center}
Hence, $S+S = S$, and every morphism is idempotent. Thus, $\textup{Hom}_{\mathcal{C}_T}(X,Y)$ is a finite lattice.
\end{proof}

Let $U$ be a $T$-module (see Section 4 of $\cite{CG}$), and denote by $F_U : \mathcal{C}_T \xrightarrow{} k\textup{-Mod}$ the functor that sends $X \in \mathcal{C}_T$ to $k[U^X]$.

If $A \in Rep(\mathcal{C}_T)$, $A$ is determined by its value on a singleton, and we have an equivalence of categories between $T$-modules and $Rep(\mathcal{C}_T)$, for a $T$-module $U$, we denote by $\Tilde{U}$ the associated additive functor in $Rep(\mathcal{C}_T)$. Via this equivalence of categories, the functor $F_U$, for a $T$-module $U$, corresponds to the linearization $k[\Tilde{U}]$.
Theorem 4.7 therefore gives us the following corollary:

\begin{corollaire}
Let $U,V$ be two $T$-modules with $V$ finite. For all $n \geq 1$, we have $\textup{Ext}^n_{\mathcal{F}(\mathcal{C}_T;k)}(F_U,F_V) = 0$.
\end{corollaire}

This result for correspondence functors (generalized or not) was not previously known. Only a few computations of $\textup{Ext}^1$ between correspondence functors (see Section 8 of \cite{BT4}) were done by Bouc and Thévenaz, using purely representation-theoretic methods as well as stabilization results for Ext.

\bibliographystyle{plain}

\nocite{*}
\bibliography{bibliographie}

\end{document}